\documentclass[reqno,11pt]{amsart}

\usepackage[T2A]{fontenc}
\usepackage{amsfonts}
\usepackage{amssymb}
\usepackage{amsmath}
\usepackage{amsthm}
\usepackage{latexsym}
\usepackage[english]{babel}
\usepackage{hyperref}
\usepackage{mathtools}

\usepackage{datetime}
\renewcommand{\dateseparator}{.}
\newcommand{\todayiso}{\twodigit\day \dateseparator \twodigit\month \dateseparator \the\year}

\newcommand{\res}{{\rm res}}

\newtheorem{theorem}{Theorem}
\newtheorem{lemma}{Lemma}[section]
\newtheorem{prop}{Proposition}[section]
\newtheorem{df}{Definition}[section]
\newtheorem{ntn}{Notation}[section]

\begin{document}

\title{\textbf{Orbifold GW theory as Hurwitz-Frobenius submanifold}}
\date{\todayiso}
\author{Alexey Basalaev}
\address{National Research University Higher School of Economics, Vavilova 7, 117312 Moscow, Russia}
\email{aabasalaev@edu.hse.ru}
\address{Leibniz Universit\"at Hannover, Welfengarten 1, 30167 Hannover, Germany}
\email{basalaev@math.uni-hannover.de}

\maketitle

\begin{abstract}
  In this paper we study the relation between the Frobenius manifolds of GW theory and Hurwitz-Frobenius manifold. We prove that orbifold GW theory of $\mathbb P^1(2,2,2,2)$ is isomorphic to the submanifold in the Hurwitz-Frobenius manifold of ramified coverings of the sphere by the genus $1$ curve with the ramification profile $(2,2,2,2)$ over $\infty$.
\end{abstract}

  \section{Introduction}

    The structure of a Frobenius manifold appears essentially in many different constructions. Examples include the base space of the semi-universal deformation of the singularity (so called Saito's flat structures), genus zero part of the GW theory, invariant theory of root systems, spaces of ramified coverings of the certain type (so called Hurwitz-Frobenius manifolds). These structures appear to be investigated in their natural environment with the stress on the special properties corresponding to the type of origin. The connection between the Frobenius manifolds of the different kind is still an open question. The best developed types of these connections is the isomorphisms between GW theory and Saito's flat structures and between Saito's flat structures and invariant theory of root systems. In this article we present a new type of correspondence. We relate GW theory of the orbifold $\mathbb P^1(2,2,2,2)$ with a submanifold in a certain Hurwitz-Frobenius manifold.

    \subsection*{Orbifold GW theory}
    In \cite{CR} the authors gave the treatment of GW-theory for an orbifold $X$. Roughly speaking their work allows us to mimic the ``usual'' GW theory to the case of an orbifold $X$. 
    
    Fix $\beta \in H_2(X, \mathbb Z)$. The authors defined the moduli space $\overline {\mathcal M}_{g,n}(X, \beta)$ of degree $\beta$ stable orbifold maps from the genus $g$ curve with $n$ marked points to $X$. Together with the suitable fundamental cycle $[ \overline {\mathcal M}_{g,n}(X, \beta) ] ^{vir}$ one can introduce the correlators like in the usual GW theory. Define ${ev_i: \overline {\mathcal M}_{g,n}(X, \beta) \rightarrow X}$ -- the map sending the stable orbifold map with $n$ markings to its value at the $i$-th marked point.

    Let $\gamma_i \in H^*_{orb} (X, \mathbb Q)$ -- the elements of the Chen-Ruan orbifold cohomology ring. The correlators are defined by:

    $$
      \langle \gamma_1 , \dots , \gamma_n \rangle_{g,n,\beta}^X := \int_{ [ \overline {\mathcal M}_{g,n}(X, \beta) ] ^{vir} } ev_1^* \gamma_1 \wedge \dots \wedge ev_n^* \gamma_n.
    $$

    It is convenient to assemble the numbers obtained into a generating function called genus $g$ potential of the (orbifold) GW theory. Let $\textbf t := \sum_i \gamma_i t_i$ for the formal parameters $t_i$ and $\{ \gamma_i \}$ -- the basis in $H^*_{orb}(X, \mathbb Q)$.

    $$
      \mathcal F^{X}_g := \sum_{n, \beta} \frac{1}{n!} \langle \textbf t , \dots , \textbf t \rangle_{g,n,\beta}^X.
    $$

    The most important for us will be the genus zero potential. Due to the geometrical properties of the moduli space of curves like in the usual GW theory the orbifold genus zero potential solves the WDVV equation too. It reads:

    \begin{equation}\label{eq:WDVV}
      \frac{\partial^3 \mathcal F_0^X }{\partial t_i \partial t_j \partial t_p} \ \eta^{pq} \ \frac{\partial^3 \mathcal F_0^X}{\partial t_q \partial t_k \partial t_l} 
      = 
      \frac{\partial^3 \mathcal F_0^X}{\partial t_i \partial t_k \partial t_p} \ \eta^{pq} \ \frac{\partial^3 \mathcal F_0^X}{\partial t_q \partial t_j \partial t_l},
    \end{equation}
    for every fixed $i,j,k,l$ and $\eta$ -- a certain bilinear form in $H^*_{orb}(X, \mathbb Q)$. An important implication of this is that $\mathcal F_0^X$ defines a \textit{formal} Frobenius manifold (we refer the reader to \cite{Manin} for details). Assume that the cohomology class $\gamma_0$ is the unity in $H^*_{orb}(X)$. The multiplication and pairing are defined by:

    \begin{equation*}
	c(\partial_i, \partial_j, \partial_k)  := \frac{\partial^3 F}{\partial t_i \partial t_j \partial t_k }, \quad \eta(\partial_i, \partial_j) := \frac{\partial^3 F}{\partial t_0 \partial t_i \partial t_j }.
    \end{equation*}
    Note that from the definition $\eta_{ij}, c_{ijk}$ are functions of $t$. A special property of the orbifold GW theory is that for the choice of $\gamma_0$ we assumed $\eta_{ij}$ does not depend on the point: $\partial_k \eta_{ij} = 0$.
    
    \subsection*{Space of ramified coverings}
    We relate it to the so-called Hurwitz-Frobenius manifolds, introduced by Dubrovin in \cite{D} (see also \cite{B}). Consider the space of meromorphic functions 
    $$
      C \xrightarrow{\lambda} \mathbb P^1
    $$
    on the compact genus $g$ Riemann surface $C$. Fix the pole orders of $\lambda$ to be $\textbf k := \{k_1, \dots, k_m\}$:
    $$
      \lambda^{-1}(\infty) = \{ \infty_1, \dots, \infty_m \}, \quad \infty_p \in C,
    $$
    so that locally at $\infty_p$ we have $\lambda(z) = z^{k_p}$. 

    Such meromorphic functions define the ramified coverings of $\mathbb P^1$ by $C$ with the ramification profile $\textbf k$ over $\infty$. We further assume that $\lambda$ has only simple ramification points at $P_i \in \mathbb P^1 \backslash \{0\}$. The degree of the ramified covering is computed to be $N = \sum k_i$ and using the Riemann-Hurwitz formula we can compute the dimension of this space of functions:
    $$
      n = 2g -2 + \sum_{i=1}^m k_i + m,
    $$
    that is exactly the number of simple ramification points. The smooth part of the Hurwitz-Frobenius manifold is parametrized by the values of $\lambda$ at the simple ramification points: $(\lambda(P_1), \dots, \lambda(P_n))$.

    \begin{df}\label{defintion: Hurwitz equivalence}
      Two pairs $(C_1, \lambda_1)$ and $(C_2, \lambda_2)$ as above are said to be Hurwitz-equivalent if $\lambda_1 = \psi \circ \lambda_2$ for some analytic map $\psi: C_1 \rightarrow C_2$.
    \end{df}

    In what follows we consider the pairs $(C, \lambda)$ up to the equivalence introduced.

    \begin{df}
      We define the Hurwitz-Frobenius manifold $\mathcal H_{g; \textbf{k}}$ to be the moduli space of pairs $(C, \lambda)$ as above with the additional data:
      \begin{itemize}
	\item $\{ a_1, \dots, a_g, b_1, \dots, b_g \}$ - the choice of a symplectic basis in $H_2(C)$,
	\item $\{ w_1, \dots, w_n \}$ - uniformization parameter of $\lambda$ at $\infty_i$
	$$
	  w_i^{k_i}(z) = \lambda(z), \quad z \in U(\infty_i).
	$$
      \end{itemize}

    \end{df}

    \subsection*{Frobenius manifold structure on $\mathcal H_{g; \textbf{k}}$}
    Following Dubrovin we define a Frobenius manifold structure on $\mathcal H_{g; \textbf{k}}$. Let $\phi$ be a differential of the first kind on $C$. Define the multi-valued coordinate $v(P)$ on $C$ as:

    \begin{equation}\label{eq:Multi-valued coordinate}
      v(P) = \int_{\infty_1}^P \phi.
    \end{equation}

    Introduce the coordinates on $\mathcal H_{g, \textbf k}$:
    
    \begin{equation*}
      \begin{aligned}
	t_{i; a} & := \res_{\infty_i} (w_i)^{-a} v d\lambda,  \quad  & 1 \ge i \ge m, k_i > a \ge 1,
	\\
	v_j & := \int_{\infty_1}^{\infty_j} \phi, \quad V_j := -\res_{\infty_j} \lambda \phi, \quad  & m \ge j > 1,
	\\
	B_j & := \oint_{b_j} \phi, \quad C_j := \oint_{a_j} \lambda \phi. & g \ge j \ge 1.
      \end{aligned}
    \end{equation*}

    Let $\partial_i$ be the basis vectors in $T {\mathcal H}_{g, \textbf k}$ w.r.t. the coordinates introduced and $\lambda^\prime = \partial_v \lambda$. Define structure constants of the multiplication $c(\cdot,\cdot,\cdot)$ and pairing $\eta(\cdot,\cdot)$:

    \begin{equation}\label{eq:MetricAndMultiplication}
      \begin{aligned}
	\eta (\partial_i, \partial_j) &:= \sum \res_{\lambda^\prime = 0} \frac{\partial_i \lambda \partial_j \lambda {\rm d}v}{\lambda^\prime},
	\\
	c(\partial_i, \partial_j, \partial_k) &:= \sum \res_{\lambda^\prime = 0} \frac{\partial_i \lambda \partial_j \lambda \partial_k \lambda {\rm d}v}{\lambda^\prime}.
      \end{aligned}
    \end{equation}

    The theorem of Dubrovin states that this multiplication and pairing define a Frobenius manifold structure on $\mathcal H_{g; \textbf k}$ with the coordinates introduced above playing the role of flat coordinates. Namely, in these coordinates we have $\partial_i \eta_{jk} = 0$. For the particular choice of flat coordinates as above the only non-vanishing entries of $\eta$ are:

    \begin{equation*}
      \eta_{t_{i;a}, t_{j;b}} = \frac{1}{k_i} \delta_{i,j} \delta_{a+b, k_i}, 
      \quad
      \eta_{v_i, V_j} = \frac{1}{k_i} \delta_{i,j},
      \quad
      \eta_{B_j, C_k} = \frac{1}{2 \pi i} \delta_{j,k}.
    \end{equation*}

    \begin{df}
      The function $\mathcal F^{H}$ called Frobenius (or WDVV) potential is defined by: 
      $$
	\partial_i \partial_j \partial_k \mathcal F^{H} = c(\partial_i,\partial_j, \partial_k).
      $$
    \end{df}

    It is clear from the definition that the multiplication defined by the structure constants $c(\partial_i,\partial_j, \partial_k)$ is commutative and associative. From the second property it follows that $\mathcal F^{H}$ is a solution of the WDVV equation \eqref{eq:WDVV}.

    \subsection*{GW to Hurwitz-Frobenius correspondence}
    Consider the space $\mathcal H_{1; (2,2,2,2)}$. Namely $g=1$ and $\textbf k = (2,2,2,2)$. We identify the genus 1 Riemann surface $C$ with the elliptic curve $\mathcal E = \mathbb C / (2 \omega_1 \mathbb Z + 2 \omega_2 \mathbb Z)$. The generic meromorphic function $\lambda: \mathcal E \rightarrow \mathbb P^1$ in this case reads:

    \begin{equation}\label{eq:LambdaGeneric}
      \lambda(z) = \sum_{i = 1}^4 \left( \wp(z- a_i) u_i + \frac{1}{2} \frac{\wp^\prime(z - a_i)}{\wp(z-a_i)} s_i \right) + c,
    \end{equation}
    where $\wp(z) = \wp(z ; 2\omega_1, 2\omega_2)$.

    We have the ``moduli'':
    \begin{itemize}
      \item $a_i$ -- positions of the poles on $\mathcal E$,
      \item $u_i,s_i$ -- behaviour at the poles,
      \item $c$ -- the shift,
      \item $2\omega_1, 2\omega_2$ -- the ``moduli'' of the elliptic curve itself.
    \end{itemize}

    The main theorem of this paper is the following:

    \begin{theorem}\label{theorem:P14}
      The Frobenius manifold of the $\rm GW$-theory of $\mathbb P^1 (2,2,2,2)$ is the Frobenius submanifold in the Hurwitz-Frobenius manifold $\mathcal H_{1,(2,2,2,2)}$ obtained by the following restriction:

      \begin{equation}\label{eq:restriction}
	\begin{aligned}
	  a_1 = 0, \quad a_2 = \omega_1 & + \omega_2, \quad a_3 = \omega_1, \quad a_4 = \omega_2,
	  \\
	  & s_1 = s_2 = s_3 = s_4 = 0.
	\end{aligned}
      \end{equation}
    \end{theorem}

    \subsection{Organization of the paper}
    In Section \ref{section: GW} we review the GW theory of $\mathbb P^1(2,2,2,2)$ explicitly writing down genus zero potential in a suitable form. In Section \ref{section: Hurwitz} we compute flat coordinates of the Hurwitz-Frobenius manifold. Section \ref{section: StrCon} describes the technique of the computation of the structure constants of the Hurwitz-Frobenius manifold. In Section \ref{section: Restriction} we consider the restriction of the structure constants to the submanifold. Finally in Section \ref{section: Conincidence} we show that up to rescaling of variables the WDVV potential of the submanifold coincides with the suitably written GW potential of $\mathbb P^1(2,2,2,2)$.

    \subsection{Acknowledgement}
    The author is grateful to Prof. Wolfgang Ebeling for the fruitful discussions.

  \section{GW-theory of $\mathbb P^1(2,2,2,2)$}\label{section: GW}
    We consider the orbifold GW theory of one particular orbifold $X = \mathbb P^1(2,2,2,2)$. This is a projective line with four points with the non-trivial orbifold structure $\mathbb Z_2$. Equivalently it could be obtained as the global quotient of an elliptic curve by a $\mathbb Z_2$ action.
    An explicit treatment of the genus zero part of its orbifold GW theory was given in \cite{ST}. We have to introduce several objects to present it here.

    \begin{df}
      The functions $\vartheta_i(z,\tau)$ for $\tau \in \mathbb H$ and $z \in \mathbb C$ represented by the following Fourier expansions:
      \begin{equation*}
	\begin{aligned}
	  \vartheta_1(z,\tau) &= i \sum_{n = -\infty}^\infty (-1)^n e^{(n-1/2)^2 \pi i \tau} e^{(2n-1) \pi i z},
	  \\
	  \vartheta_2(z,\tau) &= \sum_{n = -\infty}^\infty e^{(n-1/2)^2 \pi i \tau} e^{(2n-1) \pi i z},
	  \\
	  \vartheta_3(z,\tau) &= \sum_{n = -\infty}^\infty e^{n^2 \pi i \tau} e^{2n \pi i z},
	  \\
	  \vartheta_4(z,\tau) &= \sum_{n = -\infty}^\infty (-1)^n e^{n^2 \pi i \tau} e^{2n \pi i z}.
	\end{aligned}
      \end{equation*}
      will be called Jacobi theta functions or just theta functions.
    \end{df}

    It is clear from their Fourier expansions that Jacobi theta functions satisfy the Heat Equation:
    $$
      \frac{\partial^2 \vartheta_i(z, \tau)}{\partial z^2} = 4 \pi i \frac{\partial \vartheta_i (z, \tau)}{\partial \tau}, \quad 4 \ge i \ge 1.
    $$

    \begin{df}
      The functions $\vartheta_i(\tau) := \vartheta_i(0,\tau)$ for $4 \ge i \ge 2$ will be called theta constants.
    \end{df}
    Note that $\vartheta_1(0,\tau) \equiv 0$. Therefore we do not consider it.

    \begin{ntn}
      In what follows we skip the argument for the theta constants whenever it is fixed and we denote:
      $$
	\vartheta^\prime_i(\tau) := \frac{\partial}{\partial z} \vartheta_i (\tau).
      $$
    \end{ntn}

    \begin{df}
      Define: 
      $$
	X_i(\tau) := 2 \frac{\partial}{\partial \tau} \log \vartheta_i, \quad 2 \ge i \ge 4.
      $$
    \end{df}

    In \cite{ST} the authors computed explicitly the genus zero potential of the GW-theory of $\mathbb P^1 (2,2,2,2)$:

    \begin{align*}
      & \mathcal F ^{\mathbb P^1 (2,2,2,2)}_0 (t_0,t_1,t_2,t_3,t_4,t) = \frac{1}{2} t_0^2 t + \frac{1}{4}t_0 (\sum_{i=1}^4 t_i^2) + (t_1t_2t_3t_4)f_0(t) \\
	& + \frac{1}{4}(t_1^4+t_2^4+t_3^4+t_4^4) f_1(t) + \frac{1}{6}(t_1^2t_2^2 + t_1^2t_3^2+t_1^2t_4^2 + t_2^2t_3^2+t_2^2t_4^2+t_3^2t_4^2) f_2(t),
    \end{align*}

    where
    \begin{equation}\label{eqHalphenReduction}
    \begin{cases}
      &f_0(t) := \frac{1}{8} X_3(t) - \frac{1}{8} X_4(t), 
      \\
      &f_1(t) := -\frac{1}{12} X_2(t) - \frac{1}{48}X_3(t) - \frac{1}{48}X_4(t),
      \\
      &f_2(t) := -\frac{3}{16} X_3(t) - \frac{3}{16}X_4(t).
    \end{cases}
    \end{equation}

    The genus zero potential $\mathcal F ^{\mathbb P^1 (2,2,2,2)}_0$ satisfies the quasi-homogeneity condition. Denote by $E_{GW}$ the Euler vector field 
    $$
      E_{GW} := t_0 \frac{\partial }{\partial t_0} + \sum_{i=1}^4 t_i \frac{1}{2} \frac{\partial}{\partial t_i}.
    $$
    Then the quasihomogeneity condition reads:
    $$
      E_{GW} \cdot \mathcal F ^{\mathbb P^1 (2,2,2,2)}_0 = 2 \mathcal F ^{\mathbb P^1 (2,2,2,2)}_0.
    $$

    The WDVV equation on $\mathcal F ^{\mathbb P^1 (2,2,2,2)}_0$ is equivalent to the system of PDE on $X_i(t)$ known as Halphen's system:
    \begin{equation}\label{eqHalphen}
    \begin{cases}
      &\frac{d}{dt} (X_2(t) + X_3(t)) = 2X_2(t)X_3(t), 
      \\
      &\frac{d}{dt} (X_3(t) + X_4(t)) = 2X_3(t)X_4(t),
      \\
      &\frac{d}{dt} (X_4(t) + X_2(t)) = 2X_4(t)X_2(t).
    \end{cases}
    \end{equation}

    It is well known fact that the $X_i$ as above give solution of this system (see for example \cite{O}). We do not give the proof here because it requires some additional properties of theta constants that are not important for us.

    \begin{prop}\label{eqP2222in4}
      Applying a linear change of variables the potential $\mathcal F ^{\mathbb P^1 (2,2,2,2)}_0$ can be rewritten in the form:
      \begin{equation*}
	\begin{aligned}
	  & {\mathcal F} ^{\mathbb P^1 (2,2,2,2)}_0  (t_0, \tilde t_1, \tilde t_2, \tilde t_3, \tilde t_4,t) = \frac{t_0^2 t}{2} + \frac{t_0}{2} \sum_{i=1}^4 (\tilde  t_i)^2 - (\tilde t_1^2 \tilde t_3^2 + \tilde t_2^2 \tilde t_4^2) \frac{1}{4} X_3(t)  
	  \\
	    &- ( \tilde t_1^2 \tilde t_4^2 + \tilde t_2^2 \tilde t_3^2) \frac{1}{4} X_4(t) - ( \tilde t_3^2 \tilde t_4^2 + \tilde t_1^2 \tilde t_2^2 ) \frac{1}{4} X_2(t) - \frac{1}{16} \sum_{i=1}^4 (\tilde  t_i)^4 \gamma(t),
	\end{aligned}
      \end{equation*}
      with the Euler vector field preserved: 
      $$
	E_{GW} (t_0, \tilde t_i) = E_{GW} (t_0, t_i),
      $$
      and $\gamma(t) = \frac{2}{3} \sum X_i (t) $.
    \end{prop}
    \begin{proof}
      We apply the change of variables $ t_1 = \tilde t_4 - \tilde t_3$, $t_2 = \tilde t_4 + \tilde t_3$, $ t_3 = \tilde t_1 - \tilde t_2$, $ t_4 = \tilde t_1 + \tilde t_2$ that obviously preserves the WDVV equation. Simple computations show:
      \begin{equation*}
	    \begin{cases}
	    & \frac{1}{6} f_2(t) + \frac{1}{2} f_1(t) = -\frac{1}{24} \sum X_i = -\frac{1}{16}\gamma(t), \\
	    & \frac{2}{3} f_2(t)- f_0(t)= - \frac{1}{4} X_3, \\
	    & \frac{2}{3} f_2(t) + f_0(t) = - \frac{1}{4} X_4,\\
	    & 3 f_1(t) - \frac{1}{3}  f_2(t)= - \frac{1}{4} X_2.
	    \end{cases}
      \end{equation*}

      It is an easy computation to check that the Euler vector field is perserved too.
    \end{proof}

    \begin{ntn}
      We will denote by $\mathcal F^{\rm GW}$ the genus zero potential of the orbifold GW theory of $\mathbb P^1(2,2,2,2)$ written in the form as above.
    \end{ntn}

  \section{The space $\mathcal H _{1;(2,2,2,2)}$}\label{section: Hurwitz}

    The Hurwitz-Frobenius manifold we consider in this paper is $\mathcal H_{1;(2,2,2,2)}$. It is parametrizing the meromorphic functions of the elliptic curve $\mathcal E = \mathbb C / (2\omega_1 \mathbb Z + 2\omega_2 \mathbb Z)$, $\lambda: \mathcal E \rightarrow \mathbb P^1$ with certain additional data that was presented in the introduction.

    We will use extensively the theory of elliptic functions in our treatment.

    \subsection{Elliptic functions}

    Consider the lattice $\Lambda = 2 \omega_1 \mathbb Z + 2 \omega_2 \mathbb Z$ with $\omega_2 / \omega_1 \in \mathbb H$. We will denote by $D$ its fundamental domain.

    \begin{df}
      A meromorphic function $f$ on $\mathbb C$ is called elliptic w.r.t. the lattice $\Lambda$ if it satisfies the following periodicity properties: 
      $$
	f(z + 2\omega_1) = f(z), \quad f(z + 2\omega_2) = f(z).
      $$
    \end{df}

    Recall the Weierstrass elliptic function:
    $$
      \wp\left(z;2\omega_1,2\omega_2\right) := \frac{1}{z^{2}}+\sum_{\omega \in \Lambda\smallsetminus\left\{ 0\right\} }\left(\frac{1}{\left(z-\omega\right)^{2}}-\frac{1}{\omega^{2}}\right).
    $$
    It is obvious from the definition that $\wp$ is indeed an elliptic function. Another important example is its derivative $\wp^\prime$ that is an elliptic function with the same periods.
    
    \begin{prop}\label{prop:EllipticFunctions}
      The space of elliptic functions on the elliptic curve ${E = \mathbb C / \Lambda}$ is generated by $\wp$ and $\wp^\prime$:
      $$
	\mathcal M (\mathcal E) = \mathbb C (\wp, \wp^\prime).
      $$
    \end{prop}

    For our purposes it is helpful to rewrite the expansion of $\wp$ and $\wp^\prime$ in $z$ and $\tau := \omega_2 / \omega_1$:

    \begin{equation*}
      \begin{aligned}
	\wp(z,\tau) &= z^{-2} + \frac{1}{20}g_2(\tau) z^2 + \frac{1}{28} g_3(\tau) z^4 + O(z^6),
	\\
	\wp^\prime(z,\tau) &= -2 z^{-3} + \frac{2}{20}g_2(\tau) z + \frac{4}{28} g_3(\tau) z^3 + O(z^5),
      \end{aligned}    
    \end{equation*}
    for $g_2(\tau), g_3(\tau)$ - modular invariants of the elliptic curve.

    The connection between the two definitions of the function $\wp$ is given by the equality:

    \begin{equation}\label{eq:WP omega to tau}
      (2\omega_1)^2 \wp(z;2\omega_1,2\omega_2) = \wp \left( \frac{z}{2\omega_1}; \tau \right),
    \end{equation}
    for $\tau = \omega_2 /\omega_1$.

    Another important property of the elliptic functions is the following:
    \begin{prop}
      Let $f(z)$ be an elliptic function. Then the sum of its residues in the fundamental domain $D$ of $\Lambda$ is zero:
      $$
	\sum_{a \in D} \res_{z=a} f(z) = 0.
      $$
    \end{prop}

    \begin{df}
      The Weierstrass zeta-function is defined by:
      $$
	\zeta(z;2\omega_1, 2\omega_2) = \frac{1}{z}+\sum_{w \in \Lambda\smallsetminus\left\{ 0\right\}}\left( \frac{1}{z-w}+\frac{1}{w}+\frac{z}{w^2}\right). 
      $$
    \end{df}

    Its main property is:
    $$
      - \zeta^\prime (z;2\omega_1, 2\omega_2) = \wp(z;2\omega_1, 2\omega_2).
    $$
    Note that it is not periodic w.r.t. $2\omega_i$.

    \begin{df}
      The quasi-periods $2\eta_i$ are defined by:
      $$
	2 \eta_i = \zeta(2\omega_i + z) - \zeta(z), \quad \forall z \in \mathbb C.
      $$
    \end{df}
    
    The connection between the periods and quasi-periods of the lattice $\Lambda$ is given via the Legendre identity:
    $$
      \eta_1 \omega_2 - \eta_2 \omega_1 = \frac{\pi i}{2}.
    $$
    
    \subsection{The moduli problem}
    In our setup the function $\lambda$ is defined on $\mathcal E$, therefore it has to be an elliptic function. Due to the ramification fixed it has four order 2 poles.  
    Using a Proposition \ref{prop:EllipticFunctions} we write the generic function of this form:
    \begin{equation}\label{eqLambdaGeneric}
      \lambda(z) = \sum_{i = 1}^4 \left( \wp(z- a_i; 2\omega_1, 2\omega_2) u_i + \frac{1}{2} \frac{\wp^\prime(z - a_i; 2\omega_1, 2\omega_2)}{\wp(z-a_i; 2\omega_1, 2\omega_2)} s_i \right) + c,
    \end{equation}
    from where we have the ``moduli'':
    \begin{itemize}
      \item $a_i$ -- positions of the poles on $\mathcal E$,
      \item $u_i,s_i$ -- behaviour at the poles,
      \item $c$ -- the shift,
      \item $2\omega_1, 2\omega_2$ -- the ``moduli'' of the elliptic curve itself.
    \end{itemize}

    This sums up to 14 parameters, but they are not completely free of relations. From the Riemann-Hurwitz formula we see that the dimension of the space of such functions $\mathcal H := \{\lambda \}$ as above is 12.

    Because of being an elliptic function we have:
    $$
      \sum_{z \in D} \res_z \lambda = 0 \quad \Rightarrow \quad \sum_{i=1}^4 s_i = 0.
    $$
    We assume $s_1 = 0$.

    On the covering curve we have $\mathcal E _{(2\omega_1,2\omega_2)} \cong \mathcal E _{1, \tau}$ for $\tau = \omega_2 / \omega_1$. These two elliptic curves give equivalent ramified coverings w.r.t. the Hurwitz-equivalence (see Definition \ref{defintion: Hurwitz equivalence}).

    Because of the automorphisms of the elliptic curve moving its origin we can also assume $a_1 = 0$.

    \begin{prop}
      The Hurwitz-Frobenius manifold $\mathcal H_{1;(2,2,2,2)}$ is the space of functions $\lambda$ as above considered as functions of:
      $$
	a_2,a_3,a_4, \ s_2,s_3,s_4, \ u_1,u_2,u_3,u_4, \ \omega_2 / \omega_1.
      $$
    \end{prop}

    In what follows we denote for simplicity $\mathcal H := \mathcal H_{1;(2,2,2,2)}$ and we keep the notation $a_1$ assuming that it is equal to zero.

    \subsection{Flat coordinates}

    Following Dubrovin \cite{D} we introduce flat coordinates on the space $\mathcal H$. To do this one has to fix a certain differential on the covering curve. We take:
    $$
      \phi := {\rm d} v = \frac{{\rm d}z}{2 \omega_1},
    $$
    where $z$ is the coordinate on $\mathcal E$.

    Let $w_i^{2}(z) = \lambda(z)$, for $z \in U(a_i)$ -- unformization parameter in the small neighborhood of the pole $a_i$.

    \begin{theorem}[Dubrovin]
      The following functions are flat coordinates on $\mathcal H$:
      \begin{equation*}
	\begin{aligned}
	  t_i &= \res_{a_i} (w_i)^{-1} v {\rm d}\lambda, \quad & 1 \ge i \ge 4, 
	  \\
	  v_j &= \int_{a_1}^{a_j} {\rm d}v, \quad V_j = -\res_{a_j} \lambda {\rm d}v, \quad & 2 \ge j \ge 4,
	  \\
	  B_1 &= \int_{0}^{2 \omega_2} {\rm d}v, \quad C_1 = \int_{0}^{2 \omega_1} \lambda {\rm d}v. &
	\end{aligned}
      \end{equation*}
      
      The Euler vector field of the Frobenius structure in these coordinates is given by:
      \begin{equation}\label{eq:EulerVectorField}
	E_{\mathcal H} = C_1 \frac{\partial}{\partial C_1} + \sum \frac{1}{2} t_i \frac{\partial }{\partial t_i} + \sum V_i \frac{\partial }{\partial V_i}.
      \end{equation}

    \end{theorem}

    \begin{prop}
      The ramified covering $\lambda$ is given in flat coordinates by:
      \begin{equation}\label{eqLambdaFlat}
	\begin{aligned}
	  \lambda(z) &= \sum_{i = 2}^4 \left( \frac{1}{4}\wp \left( v - v_i , \tau \right) t_i ^2 + \frac{1}{2} \frac{\wp^\prime \left( v - v_i , \tau\right)}{\wp \left( v - v_i, \tau \right)} V_i  \right) 
	  \\
	  & + \frac{1}{4}\wp(v, \tau) t_1 ^2 + \eta_1 \omega_1 \sum_{i=1}^4 t_i ^2 + C_1.
	\end{aligned}
      \end{equation}
      
    \end{prop}
    \begin{proof}
      Using the formulae by Dubrovin we compute:
      $$
	v_i = \frac{a_i}{2 \omega_1}, \quad V_i = \frac{s_i}{2 \omega_1}, \quad B_1 = \int_{0}^{2 \omega_2} \frac{dz}{2 \omega_1} = \tau,
      $$
      where $\tau = \dfrac{\omega_2}{\omega_1}$ is the modulus of the elliptic curve.

      Compute $t_i := t_{i,1}$:
      \begin{equation*}
	  t_i = \res_{a_i} \frac{z - a_i}{2 \omega_1} \frac{z - a_i}{\sqrt{u_i}} \left( \frac{-2u_i}{(z - a_i)^3} + \text{h.o.t.} \right) = - \frac{\sqrt{u_i}}{\omega_1},
      \end{equation*}
      where the branch of the square root is fixed by the choice of the uniformization parameter $w_i$.

      Note that $\dfrac{\wp^\prime}{\wp} = \dfrac{\partial}{\partial z} \log (\wp)$. The value of $\zeta(z)$ is not defined at $z=0$, therefore we have to use the limit computing $C_1$:

      \begin{equation*}
	\begin{aligned}
	  C_1 = & \frac{1}{2 \omega_1} \lim_{\epsilon \rightarrow 0} \left[ -\sum \zeta(z - a_i) u_i + \frac{1}{2} \log \wp(z-a_i) v_i + zc \right]_{\epsilon}^{2\omega_1 - \epsilon} 
	  \\
	  = & \frac{1}{2 \omega_1} \Big( \sum \left( \zeta(-a_i) - \zeta(2\omega_1 - a_i) \right) u_i  + 2\omega_1 c
	  \\
	  & + \frac{1}{2} \left( \log \wp(2\omega_1 - a_i) - \wp(-a_i) \right) v_i \Big).
	\end{aligned}
      \end{equation*}
      Because of the periodicity of the Weierstrass functions the last line vanishes. We get:
      $$
	C_1 =  c - \frac{\eta_1 }{\omega_1} \sum_{i=1}^4 u_i.
      $$

      Using the equality \eqref{eq:WP omega to tau} we get the proposition.
    \end{proof}

    In the rest of the paper we will be working with the function $\lambda(z)$ written in flat coordinates. We will not write the variable $\tau$ all the time meaning implicitly that Weierstrass functions inside are $\wp(v , \tau)$.

  \section{Structure constants of $\mathcal H _{1;(2,2,2,2)}$}\label{section: StrCon}
  
    In this section we provide all the computations needed to prove the theorem. Basically we compute structure constants of $\mathcal H _{1;(2,2,2,2)}$ using the formulae \eqref{eq:MetricAndMultiplication}.

    \subsection{Technique}
    
    In the majority of residues we have to compute we will be dealing with elliptic functions. These will be the cases when the derivative of $\lambda$ -- elliptic function itself -- is an elliptic function too. When it is so we can consider the residues at the points $v_i$ instead of looking for points where $\lambda^\prime = 0$.

    \begin{prop}\label{prop:SumOfTheResidues}
      Let $f(v)$ be an elliptic function and $x_i$ -- its set of poles such that $\lambda^\prime (x_i) \neq 0$. Then we have:
      $$
	\sum_{y: \ \lambda^\prime(y) = 0} \res_{v = y} \frac{f(v)  {\rm d}v }{\lambda^\prime(v)} = - \sum_i \res_{v = x_i} \frac{f(v) {\rm d}v }{\lambda^\prime(v)}.
      $$
    \end{prop}
    \begin{proof}
      The poles of the function $\frac{f(v)}{\lambda^\prime(v)}$ w.r.t. $v$ are:$\{x_i\} \sqcup \{y: \lambda^\prime (y) = 0 \}$. The quotient $f(v) / \lambda^\prime (v)$ is an elliptic functions and we have:
      $$
	\sum_i \res_{x_i} \frac{f(v)  {\rm d}v }{\lambda^\prime(v)} + \sum_{y: \ \lambda^\prime(y) = 0} \res_{v = y} \frac{f(v)  {\rm d}v }{\lambda^\prime(v)} = 0,
      $$
    \end{proof}
    
    The case when we can not apply this principle is $\partial_\tau \lambda$. For this we use the lemma due to Frobenius-Stickelberger \cite{FS}:

    \begin{lemma}\label{lemma:FS} Let $f(z; 2\omega_1, 2\omega_2)$ be elliptic functions with the periods $(2\omega_1, 2\omega_2)$, then the following function is elliptic too with the same periods:
    $$
      \eta_1 \frac{\partial f}{\partial \omega_1} + \eta_2 \frac{\partial f}{\partial \omega_2} + \zeta \frac{\partial f}{\partial z},
    $$
    where $\zeta = \zeta(z; 2\omega_1, 2\omega_2)$.
    \end{lemma}
    \begin{proof}
      We give a brief proof.

      Differentiating the equality $f(z + 2 \omega_1) = f(z)$ w.r.t. $\omega_1$ we get: 
      $$
	\frac{\partial}{\partial \omega_1} f(z + 2 \omega_1) + 2 \frac{\partial}{\partial z} f(z + 2\omega_1) = \frac{\partial}{\partial \omega_1} f(z).
      $$
      Together with the expression of the quasi-period we have:
      $$
	\begin{aligned}
	  & \eta_1 \frac{\partial f(z)}{\partial \omega_1} + \eta_2 \frac{\partial f(z)}{\partial \omega_2} + \zeta(z) \frac{\partial f(z)}{\partial z} = \eta_1 \frac{\partial f(z + 2 \omega_1)}{\partial \omega_1} 
	  \\
	  & + 2 \eta_1 \frac{\partial f(z + 2\omega_1)}{\partial z} + \eta_2 \frac{\partial f(z + 2 \omega_1)}{\partial \omega_2} + (\zeta(z + 2 \omega_1) - 2\eta_1) \frac{\partial f(z + 2 \omega_1)}{\partial z}
	\end{aligned}
      $$
      $$
      = \eta_1 \frac{\partial f(z + 2\omega_1)}{\partial \omega_1} + \eta_2 \frac{\partial f(z + 2\omega_1)}{\partial \omega_2} + \zeta(z + 2\omega_1) \frac{\partial f(z + 2\omega_1)}{\partial z}.
      $$
    \end{proof}

    Consider the function $f(z; 2 \omega_1, 2 \omega_2)$, applying the change of variables as in \eqref{eq:WP omega to tau} we get for $f(v, \tau)$:
    $$
      \eta_1 \frac{\partial f}{\partial \omega_1} + \eta_2 \frac{\partial f}{\partial \omega_2} + \zeta \frac{\partial f}{\partial z} = - 2\pi i \ \partial_\tau f + \zeta \partial_v f - 2\eta_1 \partial_v f.
    $$
    where we used the Legendre identity.
    
    \begin{ntn} \label{notation: h_lambda}
      Introduce the notation for the correponding elliptic function: 
      $$
	h_f(z,t) := - 2\pi i \ \partial_\tau f + \zeta \partial_v f - 2\eta_1 \partial_v f.
      $$
    \end{ntn}

  \section{Restriction of the potential}\label{section: Restriction}
  
    \begin{df}
      Define by $\mathcal F_R$ the potential obtained by the restriction of $\mathcal F^H$ of the Hurwitz-Frobenius manifold to the submanifold defined by \eqref{eq:restriction}:
      
      \begin{equation*}
	  \mathcal{F_R} := \mathcal F^H \mid_{\mathcal A},
      \end{equation*}
      for
      \begin{equation*}
	\begin{aligned}
	\mathcal A := \Big\{ v_1 = 0, \quad & v_2 = \frac{\tau}{2} + \frac{1}{2}, \quad v_3 = \frac{1}{2}, \quad v_4 = \frac{\tau}{2},
	  \\
	  & V_2 = V_3 = V_4 = 0 \Big\}.
	\end{aligned}
      \end{equation*}
    \end{df}

    \begin{df} Let $\wp(z) = \wp(z; 2\omega_1, 2\omega_2)$. The numbers $e_1,e_2,e_3 \in \mathbb C$ are defined by:
      $$
	e_1 := \wp(\omega_1), \quad e_2 := \wp(-\omega_1 - \omega_2), \quad e_3 := \wp(\omega_2).
      $$
    \end{df}

    A well-known fact from the elliptic curves theory is that:
    \begin{prop}
      The points $\omega_1$, $\omega_2$ and $\omega_1 + \omega_2$ are all zeroes of $\wp^\prime(z)$ in the fundamental domain.
    \end{prop}

    \begin{prop}\label{proposition: RestrictedPotential}
      The summands of $\mathcal F^H$ including variables $v_k$ and $V_k$ do not contribute to the restricted potential $\mathcal F_R$.
    \end{prop}
    
    We prove the proposition by computing the structure constants of the Frobenius structure. 

    From the Euler vector field of $\mathcal H$ we know that the variable $V_k$ is given a non-zero integer degree. Hence it contributes to the potential $\mathcal F^H$ polynomially. Namely there is natural number $N$ such that $V_k^n$ for $n \ge N$ does not appear in the series expansion of $\mathcal F^H$.

    It is obvious from the structure constants residue formula that $\mathcal F^H$ is well defined at $V_k = 0$. 
    
    Hence we only have to take care of the variable $v_k$ that has degree 0 and could give a non-zero contribution to the restricted potential.
    
    \begin{ntn}
      Let $f(v) = \sum_{-\infty}^\infty a_i v^i$  be formal power series in $v$, and $k \in \mathbb Z$. Denote by:
      $$
	[v^k] \ f(v) := a_k.
      $$
    \end{ntn}

    We need first the lemma:
    
    \begin{lemma}
      In flat coordinates we have the following expressions for the structure constants. Assume $i \neq k$:
      \begin{equation*}
	\begin{aligned}
	  c(t_i,v_i,v_i) &= \frac{g_2(\tau)}{20} \frac{t_i}{2}\eta_1 \omega_1 V_i,
	  \\
	  c(t_i,t_i,v_k) &= \frac{1}{8} \wp^\prime (a_k - a_i) t_i^2 + \frac{1}{4}\frac{\wp^{''}(z - a_i) \wp(z - a_i) - \left(\wp^\prime (z - a_i) \right)^2 }{\wp(z - a_i)^2} V_i
	  \\
	  c(t_i,t_i,v_i) &= 0,
	  \\
	  c(v_k,  v_k, C_1) &= 0.
	\end{aligned}
      \end{equation*}
    \end{lemma}

    \begin{proof}

      The derivative $\dfrac{\partial \lambda}{\partial v_i}$ reads:
      \begin{equation*}
	\begin{aligned}
	  \frac{\partial \lambda}{\partial v_i} &= -\frac{1}{4} \wp^\prime (v - v_i) t_i^2 - \frac{1}{2}\frac{\wp^{''}(v - v_i) \wp(v - v_i) - \left(\wp^\prime (v - v_i) \right)^2 }{\wp(v - v_i)^2} V_i
	  \\
	  &=  \frac{1}{2} \frac{t_i^2}{(v - v_i)^3} - \frac{V_i}{(v - v_i)^2} + O(1).
	\end{aligned}
      \end{equation*}
	  
      It is clear that $\dfrac{\partial \lambda}{\partial v_i}$ is an elliptic function.
		  
      \subsubsection*{Structure constants $c(t_i,v_i,v_i)$}
	  By definition we have:
	  \begin{equation*}
	      c(t_i,v_i,v_i) = - \res_{v_i} \frac{(\partial_{v_i} \lambda)^2 (\partial_{t_i} \lambda) {\rm d}v }{\lambda^\prime}.
	  \end{equation*}
	  
	  Note that the behaviour of the functions $\lambda^\prime$ and $-\partial_{v_i} \lambda$ in the neighbourhood of the point $a_i$ coincide:
	  
	  \begin{equation*}
	    \begin{aligned}			
		c(t_i,v_i,v_i) &= \res_{v_i} (\partial _{v_i} \lambda \ \partial _{t_i} \lambda) {\rm d}v
		\\
		&= [(v - v_i)] \partial _{v_i} \lambda \cdot [(v - v_i)^{-2}] \partial _{t_i} \lambda
		\\
		&+ [(v - v_i)^{-3}] \partial _{v_i} \lambda \cdot [(v - v_i)^{2}] \partial _{t_i} \lambda
		\\
		&+ \frac{2V_i}{t_i^2} [(v - v_i)^{-3}] \partial _{v_i} \lambda \cdot [(v - v_i)] \partial _{t_i} \lambda
	    \end{aligned}
	\end{equation*}
	      
	The first two lines sum to zero (basically because ${\res_{v_i} \wp^\prime(v - v_i) \wp(v - v_i) = 0}$) and from the Laurent expansion of $\wp$ we get:
	$$
	  c(t_i,v_i,v_i) = \frac{g_2(\tau)}{20} \frac{t_i}{2}\eta_1 \omega_1 V_i.
	$$
	  
      \subsubsection*{Structure constants $c(t_i,t_i,v_k)$}
	For $k \neq i$ we have:
	$$
	  c(t_i,t_i,v_k) = - \res_{v_i} \frac{(\partial _{t_i} \lambda)^2(\partial _{v_k} \lambda) {\rm d}v}{\lambda^\prime}
	    = \frac{2 }{t_i ^2} [(v-v_i)^{-4}] (\partial _{t_i} \lambda)^2 \partial _{v_k} \lambda.
	$$
	The function $\partial _{v_k} \lambda$ is regular at the point $v_i$ for $i \neq k$ and we write:
	  
	\begin{equation*}
	    c(t_i,t_i,v_k) = \frac{2 }{t_i ^2} \partial _{v_k} \lambda \mid_{v = v_i} [(v-v_i)^{-4}] (\partial _{t_i} \lambda)^2 = \frac{2 }{t_i ^2} \frac{t_i^2}{4} \partial _{v_k} \lambda \mid_{v = v_i}.
	\end{equation*}
	  
      \subsubsection*{Structure constants $c(t_i,t_i,v_i)$}
	Compute the residue for $k = i$:
	$$
	  c(t_i,t_i,v_i) = - \res_{v_i} \frac{(\partial _{t_i} \lambda)^2(\partial _{v_i} \lambda) {\rm d}v}{\lambda^\prime}
	    = \res_{v_i}  (\partial _{t_i} \lambda)^2.
	$$
	The Laurent expansion of $\partial _{t_i} \lambda$ contains even degrees of $v - v_i$ only. hence the residue vanishes.
    
      \subsubsection*{Structure constants $c(C_1,v_k,v_k)$}

	We do not need to compute the residue for this structure constants because we have:
	$$
		c(C_1,v_k,v_k) = \eta(v_k,v_k) = 0,
	$$
	where we used the equalities for the metric in the flat coordinates.
			  
	The Lemma is proved.
      \end{proof}

      Note that by the choice of $a_i$ in the restriction we have to express $\partial _{v_k} \lambda$ at one of the fundamental rectangle edge middle points:
      $$
	\begin{aligned}
	  a_2 - a_1 = \omega_1 + \omega_2, \ a_3 - a_1 = \omega_1, \ a_4 - a_1 = \omega_2, 
	  \\ 
	  a_2 - a_3 = \omega_2, \ a_2 - a_4 = \omega_1, \ a_3 - a_4 = \omega_1 - \omega_2.
	\end{aligned}
      $$

    \begin{ntn}\label{notation:e_ij}
      For $i \neq j$, $4 \ge i,j \ge 1$ denote:
      $$
	\{ 13 \} = \{ 24 \} := 1, \quad \{12\} = \{34\} := 2, \quad \{23\} = \{14\} := 3.    
      $$
    \end{ntn}

    In this notation we have:
    $$
      e_{\{13\}} = e_{\{24\}} = e_1, \quad e_{\{12\}} = e_{\{34\}} = e_2, \quad e_{\{23\}} = e_{\{14\}} = e_3.
    $$

    \begin{proof}[Proof of the Proposition \ref{proposition: RestrictedPotential}]    

      We show that all the structure constants listed above vanish under the restriction. 
      Note that we did not compute the structure constants $c(v_k,v_k,\tau)$ and $c(v_i,v_j,v_k)$. This is not needed because due to the homogeneity condition on the Hurwitz-Frobenius mnanifold potential the variables $v_k$ and $\tau$ have degree zero. Therefore the summands of ${\mathcal F}^H$ giving these structure constants appear with the factor of other variables that are assigned non-zero degree. These are $V_k$ and $t_k$. Therefore these summands contribute to the structure constants $c(t_k, \cdot, \cdot)$ or $c(V_k, \cdot, \cdot)$ whose vanishing we prove.

      It is clear that all the summands that have a factor of $V_i$ vanish. There are only two structure constants that we have to treat more carefully: $c(t_i,t_i,v_k)$ and $c(\tau, v_k,v_k)$. For the first one we have:
      \begin{equation*}
	\begin{aligned}
	  c(t_i,t_i,v_k) = - \frac{1}{2} \partial _{v_k} \lambda (v_i - v_k).
	\end{aligned}
      \end{equation*}

      The points $a_i - a_j$ are precisely those where $\wp^\prime(z; 2\omega_1, 2\omega_2)$ vanishes. And we get:
      $$
      \partial _{v_k} \lambda (a_i - a_k) = \frac{1}{2} \frac{\wp^{''} \left( \frac{1}{2 \omega_1} (a_i - a_j)  \right) }{e_{\{ik\}}} V_i.
      $$
      This expression vanishes by setting $V_i = 0$.
      
    \end{proof}
  
  \section{Proof of Theorem \ref{theorem:P14}}\label{section: Conincidence}

    To prove Theorem \ref{theorem:P14} we compute the structure constants in the variables $t_k$, $C_1$ and $\tau$.

    \subsection{Structure constants including variables $t_k$, $C_1$ and $\tau$ only}

    \begin{prop} In flat coordinates we have:
      \begin{equation*}
	\begin{aligned}
	  & c(\tau, C_1,C_1) = \frac{1}{2 \pi i}, 
	  \\
	  & c(t_i, t_i,C_1) = \frac{1}{2},
	  \\
	  & c(t_i, t_i, t_i) = 3 t_i \omega_1 \eta_1,  
	  \\
	  & c(t_i, t_i, t_j) = t_j \left( \frac{1}{4}\wp (a_i - a_j) + \eta_1 \omega_1 \right).
	\end{aligned}
      \end{equation*}
    \end{prop}
    \begin{proof}  
      ~
      \subsubsection*{Structure constant $c(\tau, C_1,C_1)$}
	By definition we have:
	$$
	      c(\tau, C_1,C_1) = \sum \res_{\lambda^\prime = 0} \frac{\partial_\tau \lambda (v) {\rm d}v}{\lambda^\prime(v)},
	$$

	Apply Lemma \ref{lemma:FS}:
	$$
	      c(\tau, C_1,C_1) = - \frac{1}{2 \pi i} \sum \res_{\lambda^\prime = 0} \frac{h_\lambda(v) {\rm d}v}{\lambda^\prime (v)}
	$$
	where we used in the last equation that the $\zeta$-function has only one pole at $v=0$. The function $h_\lambda$ is elliptic and we can apply the proposition \ref{prop:SumOfTheResidues}:
	$$
	      c(\tau, C_1,C_1) = \frac{1}{2\pi i} \sum \res_{v_p} \frac{h_\lambda {\rm d}v}{\lambda^\prime} = \frac{1}{2\pi i} \sum \res_{v_p} \frac{ \zeta \lambda^\prime - 2 \pi i \partial_\tau \lambda - 2 \eta_1 \lambda^\prime }{\lambda^\prime} {\rm d}v.
	$$
	The function $\partial_\tau \lambda / \lambda^\prime$ is regular at the points $v_p$ and does not contribute to the residue:
	$$
	      c(\tau, C_1,C_1) = \frac{1}{2 \pi i}\sum \res_{v_p} \zeta {\rm d}v = \frac{1}{2 \pi i} \res_{v_1} \zeta {\rm d}v = \frac{1}{2 \pi i}.
	$$

	\subsubsection*{Structure constant $c(t_i, t_i,C_1)$}

	\begin{equation*}
	  \begin{aligned}
	    c(t_i, t_i,C_1) &= \sum \res_{\lambda^\prime = 0} \frac{(\partial_{t_i} \lambda)^2 \partial_{C_1} \lambda {\rm d}v}{ \lambda^\prime}
	    \\
	    & = - \res_{v_i} \frac{\left( \dfrac{2t_i }{4(v - v_i)^2} + h.o.t. \right)^2 {\rm d}v}{\dfrac{- 2 t_i^2 }{4 (v - v_i)^3} + h.o.t.}
	    \end{aligned}
	\end{equation*}
	where we use $h.o.t.$ for the higher order terms.

	\begin{equation*}
	    c(t_i, t_i,C_1) = \frac{t_i^2}{4} \frac{2}{t_i^2} = \frac{1}{2}.
	\end{equation*}

      \subsubsection*{Structure constant $c(t_i, t_i, t_i)$}
    
	\begin{equation*}
	  \begin{aligned}
	    c(t_i, t_i, t_i) &= - \res_{v_i} \frac{ \left( \partial_{t_i} \lambda \right)^3 {\rm d}v}{\lambda^\prime}
	    \\
	    & = \frac{2}{t_i^2 }[(v - v_i)^{-4}] (\partial_{t_i} \lambda)^3.
	  \end{aligned}
	\end{equation*}

	The Taylor expansion of the functions in the numerator is:
	$$
	  \partial_{t_i} \lambda = \frac{t_i}{2}  \left( \frac{1}{(v - v_i)^2} +  4 \eta_1 \omega_1 + O \left((v - v_i)^2 \right) \right).
	$$
	There are only two options to get a degree $-4$ factor from its third power. Distributed in three factors they read: $-1,-1,-2$ and $-2,-2,-0$. The first one is not possible because degree -1 in $v - v_i$ appears only as the multiple of the variable $V_i$.

	\begin{equation*}
	    c(t_i, t_i, t_i) = \frac{2}{t_i^2} \frac{3 t_i^3}{4} 2 \eta_1 \omega_1 = 3 t_i \omega_1 \eta_1.
	\end{equation*}

      \subsubsection*{Structure constant $c(t_i, t_i, t_j)$}

	\begin{equation*}
	  \begin{aligned}
	    c(t_i, t_i, t_j) &= - \res_{v_i} \frac{ \left( \partial_{t_i} \lambda \right)^2 \left( \partial_{t_j} \lambda \right) {\rm d}v }{\lambda^\prime}
	    \\
	    & = \frac{2}{t_i^2}[(v - v_i)^{-4}] (\partial_{t_i} \lambda)^2 \left( \partial_{t_j} \lambda \right).
	  \end{aligned}
	\end{equation*}

	The factor $\partial_{t_j} \lambda$ is regular at the point $v_i$. Therefore we just take the value of it at the point $v_i$.

	\begin{equation*}
	    c(t_i, t_i, t_j) = \frac{2}{t_i^2} \frac{t_i^2}{4}   \left( \partial_{t_j} \lambda \right) \mid_{v = v_i}.
	\end{equation*}

    \end{proof}

    \subsection{Theta constants and elliptic functions}

      Let $e_i$ be values of $\wp(v, \tau)$ at the period rectangle edges middle points. The quantities $e_i$ are expressed via the theta constants (\cite{L}, section 6 \footnote{Note the difference in the $z$ coordinate normalization of \cite{L} with ours.}):

      \begin{equation*}
	\begin{aligned}
	  e_1 =  
	    \frac{1}{3} \frac{\vartheta_1^{\prime\prime\prime}}{\vartheta_1^\prime} - \frac{\vartheta_{2}^{\prime\prime}}{\vartheta_{2}},
	  \\
	  e_2 =  
	    \frac{1}{3} \frac{\vartheta_1^{\prime\prime\prime}}{\vartheta_1^\prime} - \frac{\vartheta_{3}^{\prime\prime}}{\vartheta_{3}},
	  \\
	  e_3 = 
	    \frac{1}{3} \frac{\vartheta_1^{\prime\prime\prime}}{\vartheta_1^\prime} - \frac{\vartheta_{4}^{\prime\prime}}{\vartheta_{4}}.
	\end{aligned}
      \end{equation*}
    
      Using the heat equation we get:
      $$
	\frac{\vartheta_{p}^{\prime \prime}}{\vartheta_{p} } = 2 \pi i  \frac{\partial_\tau \vartheta_{p}}{\vartheta_{p}} = 2 \pi i X_{p}.
      $$

      We will also use the expression:
      $$
	\eta_1 \omega_1  = - \frac{1}{12} \frac{\vartheta_1^{\prime\prime\prime}}{\vartheta_1^\prime}.
      $$
      An important property of the derivatives of the theta constants is the following:
      $$
	\frac{\vartheta_{1}^{\prime \prime \prime}}{\vartheta_{1} } = \frac{\vartheta_2^{\prime \prime}}{\vartheta_2} + \frac{\vartheta_3^{\prime \prime}}{\vartheta_3} + \frac{\vartheta_4^{\prime \prime}}{\vartheta_4}.
      $$
      Using it together with the heat equation we get:
      $$
	\omega_1 \eta_1 = -\frac{1}{12} \sum_{p=2}^4 \frac{\vartheta_p^{\prime \prime}}{\vartheta_p } = - \frac{\pi i}{6} \sum_{p=2}^4 X_p = - \frac{\pi i}{4} \gamma(\tau).
      $$

    \subsection{Structure constants at the special point}

      \begin{prop} For the potential $\mathcal F_R$ we have:
	  \begin{equation*}
	    \frac{\partial^3 \mathcal F_R}{(\partial t_i)^2 \partial t_j} = - t_j \frac{\pi i}{2} X_{ij}(\tau).
	  \end{equation*}
      \end{prop}
      \begin{proof}
	Because of Proposition \ref{proposition: RestrictedPotential} the potential $\mathcal F_R$ is obtained by integrating the structure constants $c(\cdot,\cdot,\cdot)$ of $\mathcal H$ containing $t_i, \tau, C_1$ only. We have:
	$$
	  \frac{\partial^3 \mathcal F_R}{(\partial t_i)^2 \partial t_j} = \frac{\partial^3 \mathcal F^H}{(\partial t_i)^2 \partial t_j} \mid_{\mathcal A} = \int c(t_i,t_i,t_j) dt_i^2 dt_j.
	$$

	The latter one can be computed using theta constants. For the structure constant under the integral we have:
	\begin{equation*}
	  \begin{aligned}
	    \left( \frac{1}{4}\wp (v_i - v_j) + \eta_1 \omega_1 \right) &= 
	      \left( 
		\frac{1}{12} \frac{\vartheta_1^{\prime\prime\prime}}{\vartheta_1^\prime} - \frac{1}{4} \frac{\vartheta_{\{ij\}}^{\prime\prime}}{\vartheta_{\{ij\}}} 
	      \right)
	      -\frac{1}{12} \frac{\vartheta_1^{\prime\prime\prime}}{\vartheta_1^\prime}
	    \\
	    &= - 
	      \frac{1}{4}
	      \frac{\vartheta_{\{ij\}}^{\prime\prime}}{\vartheta_{\{ij\}}}.
	  \end{aligned}
	\end{equation*}
	Where we used the convention of Notation \ref{notation:e_ij} in the double-index subscript.

	Using the heat equation for $\vartheta_{\{ij\}}$ we get the proposition.
      \end{proof}

      \subsection{Restricted potential}
      Integrating the structure constants that we have computed we write down the potential of the submanifold obtained by \eqref{eq:restriction}.
      It reads:
      \begin{equation*}
	\begin{aligned}
	  F &= \frac{C_1^2 \tau}{2} \frac{1}{2 \pi i}  + C_1 \sum_i \frac{t_i^2}{4} 
	  \\
	  & - \sum_{i > j} \frac{t_i^2 t_j^2}{4} \frac{\pi i}{2} X_{ij}(\tau) - \sum_i \frac{t_i^4}{24} \frac{3 \pi i}{4} \gamma(\tau).
	\end{aligned}
      \end{equation*}

      Introducing $t_0 := C_1 / \sqrt{2}$ and making the change of variables $\tilde t_i = t_i / \sqrt[4]{2}$ the potential of the restricted submanifold reads:
      \begin{equation*}
	\begin{aligned}
	  F &= \frac{t_0^2 \tau}{2 \pi i} + t_0 \sum_i \frac{t_i^2}{2} - \sum_{i > j} \frac{t_i^2 t_j^2}{4} \pi i X_{ij}(\tau)  - \sum_i \frac{t_i^4}{16} \pi i \gamma(\tau).
	\end{aligned}
      \end{equation*}
      
      Consider also the change of variables $\tilde \tau = \tau/(\pi i)$. It is obvious from the form of the Halphen's system \eqref{eqHalphen} that if the functions $X_i(\tau)$ give the solution triple, then the functions $ \frac{1}{\pi i} X_i(\frac{\tau}{\pi i})$ solve it too giving the same solution. Hence under this change of variables the potential transforms to one written in the form of Proposition \ref{eqP2222in4}. Theorem \ref{theorem:P14} is proved.

\end{document}